\documentclass{amsart}
\allowdisplaybreaks[1]
\usepackage{mathtools} \mathtoolsset{showonlyrefs}
\usepackage{graphicx}
\usepackage{enumitem}
\usepackage{verbatim}

\newtheorem{theorem}{Theorem}
\newtheorem{remark}[theorem]{Remark}
\newtheorem{proposition}[theorem]{Proposition}
\newtheorem{lemma}[theorem]{Lemma}
\def \R{\mathbb R}
\newcommand \vol[2][2]{\left|#2\right|_{#1}}
\def \RpK{{R_pK}}

\title{Planar radial mean bodies are convex}
\author{J. Haddad}
\begin{document}
\maketitle

\begin{abstract}
	The radial mean bodies of parameter $p>-1$ of a convex body $K \subseteq \R^n$ are radial sets introduced in \cite{GZ98} by Gardner and Zhang.
	They are known to be convex for $p\geq 0$.
	We prove that if $K \subseteq \R^2$ is a convex body, then its radial mean body of parameter $p$ is convex for every $p \in (-1,0)$.
\end{abstract}

\section{Introduction}
\label{sec_introduction}

Let $K \subseteq \R^n$ be a convex body (a compact convex set with non-empty interior).
The radial mean body of $K$ of parameter $p > 0$ is the unit ball of $\R^n$ of the norm defined by
\begin{equation}
	\label{eq_norm_Lp}
	\|v\|_{\RpK} = \left(\frac 1{\vol[n]{K}} \int_K \varrho_{K}(x,v)^p dx \right)^{-1/p}, v \in \R^n
\end{equation}
where $\vol[n]{K}$ is the $n$-dimensional volume of $K$ and $\varrho_{K}(x,v)$ is the largest $\lambda > 0$ such that $x+\lambda v \in K$, this is, the radial function of $K$ with respect to $x \in K$ in the direction of $v$.

The family of bodies $\RpK$ was introduced and studied by Gardner and Zhang in \cite{GZ98} where several important properties and inequalities were established. In particular, the results in \cite{GZ98} include the Zhang inequality on the polar projection body (see \cite{Zhang91}) and the Rogers-Shephard inequality (see \cite{RS57}).
Its close relation with the X ray transform, the covariogram function, the convolution bodies and Berwald-type inequalities on logarithmically concave functions makes $\RpK$ a very interesting geometric object.
These subjects were studied further in \cite{alonso2024brunn}, \cite{Tsol}, \cite{alonso2020extension} and \cite{HL22} to cite just a few examples.

The definition of $\RpK$ also makes sense in the range $p \in (-1,0)$, while for $p=0$ it is defined via continuity (see \cite{GZ98} or Proposition \ref{prop_GZXray} below).
Even without knowing if the right-hand side of \eqref{eq_norm_Lp} is a norm we still denote it by $\|x\|_{\RpK}$, and the set $\RpK$ can be defined as its level set
\[
	\RpK = \left\{ x \in \R^n : \|x\|_{\RpK} \leq 1 \right\}, \qquad p > -1
\]
which is a star set with respect to the origin.

The non trivial fact that $\RpK$ is convex was proven in \cite[Section 4]{GZ98} for $p \geq 0$, by generalizing a previous result by K. Ball \cite{Ball88}, on logarithmically concave functions.
However, the proof breaks down if $p \in (-1,0)$, and the convexity of $\RpK$ remains an open problem.
It seems that there is no hope of adapting the proof given in \cite{GZ98} to this range, and after more than two decades, little to no progress was made in this problem.
In this paper we prove the convexity of $\RpK$ in the plane.
\begin{theorem}
	\label{thm_main}
	Let $K \subseteq \R^2$ be a convex body, then $\RpK$ is convex.
\end{theorem}

The proof will be divided in Sections \ref{sec_polygons} and \ref{sec_cones}. In Section \ref{sec_polygons} we find an interesting formula for the norm \eqref{eq_norm_Lp}, for a generic family of polygons $K$ (Proposition \ref{prop_structure_f}), and prove its convexity in a cone of $\R^2$. Proposition \ref{prop_structure_f} is the core of our result. Then in Section \ref{sec_cones} we extend the convexity to a finite set of non-overlapping cones covering the plane, and analyze the behaviour of the norm at the intersection of these cones.

\subsection*{Acknowledgments}
The author was supported by Grant RYC2021 - 031572 - I, funded by the Ministry of Science and Innovation / State Research Agency / 10.13039 / 501100011033 and by the E.U. Next Generation EU/Recovery, Transformation and Resilience Plan, and by Grant PID2022-136320NB-I00 funded by the Ministry of Science and Innovation.

\section{Preliminaries}
\label{sec_preliminaries}

In this paper a {\it closed cone outside the origin} is a set of the form 
\[\{x \in \R^2 \setminus \{0\}: \langle x, v_1 \rangle \geq 0 \text{ and } \langle x, v_2 \rangle \geq 0 \}\]
where $v_1, v_2$ are non-zero vectors.
A closed cone outside the origin is a closed subset of $x \in \R^2 \setminus \{0\}$.
An {\it open cone} is defined similarly with strict inequalities, and automatically does not contain the origin so we can avoid the term {\it outside the origin}.

A non-negative function $f:\R^n \to \R$ is {\it homogeneous of degree $\alpha \in \R$} if $f(\lambda x) = |\lambda|^\alpha f(x)$ for every $\lambda \in \R, x \in \R^n$.

For $x \in \R^n \setminus \{0\}$ the line generated by $x$ is denoted by $\langle x \rangle$, and its orthogonal complement, by $\langle x \rangle^\perp$.
For $K \subseteq \R^n$ a convex body and $y \in \R^n$, the {\it X ray of $K$} in the direction of $v$ going through $y$ is the segment $K \cap (y+\langle v \rangle)$. Its length is denoted by $X_v K(y)$.
We will use the following well known formula for $\|\cdot\|_{\RpK}$ in terms of the X rays of $K$.
\begin{proposition}[Theorem 2.2 in {\cite{GZ98}}]
	\label{prop_GZXray}
	Let $K$ be a convex body and $p > -1$. For $x\in \R^n$ a unitary vector,
	\begin{equation}
		\label{eq_norm_Xrays}
		\|x\|_{\RpK} = \left((p+1)\vol[n]{K}  \int_{\langle x \rangle^\perp} X_x K(y)^{1+p} dy \right)^{-1/p}.
	\end{equation}
\end{proposition}

Since we will prove Theorem \ref{thm_main} for polygons first, and then use approximation, we need to show that $\RpK$ is continuous with respect to $K$.
\begin{lemma}
	\label{lem_approximation}
	Let $K_m \subseteq \R^n$ be a sequence of convex bodies converging to the convex body $K \subseteq \R^n$ in the Hausdorff metric of compact sets.
	Then for every $p > -1$, $\|\cdot\|_{R_pK_m}$ converges pointwise to $\|\cdot\|_{\RpK}$.
\end{lemma}
\begin{proof}
	Fix $x$ a unitary vector and $y \in \langle x \rangle^\perp$.
	If $y+\langle x \rangle$ does not intersect $K$, then by the compactness of $K$, the distance from $y+\langle x \rangle$ to $K$ is strictly positive, and $y+\langle x \rangle$ is disjoint from $K_m$ for sufficiently large $m$.
	If $y+\langle x \rangle$ intersects the interior of $K$, consider a point $p_0 \in y+\langle x \rangle$ in the interior of $K$.
	For $m$ large enough, $p_0$ belongs to $K_m$.
	The convergence in the Hausdorff metric of convex bodies then implies $\varrho_{K_m}(p_0, \pm x) \to \varrho_K(p_0, \pm x)$.
	This means that $X_x K_m (y) \to X_x K(y)$ provided that $y + \langle x \rangle$ is not a support line of $K$.

	Since the orthogonal projection of $K$ onto $\langle x \rangle^\perp$ has a relative boundary of zero $n-1$ dimensional measure, we get that $X_xK_m$ converges pointwise to $X_x K$ in almost every point of $\langle x \rangle^\perp$.

	The sets $K_m, K$ are uniformly bounded by the convergence and the compactness of $K$.
	Then applying the Dominated Convergence Theorem we deduce that $\|x\|_{R_pK_m} \to \|x\|_{\RpK}$.
	By the homogeneity of $\|\cdot\|_\RpK$, the same holds for all $x \in \R^n$ and the theorem follows.
\end{proof}

In several situations it will be easier to show the convexity of sets and functions, locally.
We use the following result by Tietze on ``local to global'' convexity.
\begin{lemma}[Satz 1 in \cite{tietze1928konvexheit}]
	\label{lem_locallyconvex}
	Let $K \subseteq \R^n$ be a closed connected set.
	Then $K$ is convex if and only if for every $x \in K$ there exists $\varepsilon > 0$ such that $B(x, \varepsilon) \cap K$ is convex.
\end{lemma}

We say that $K \subseteq \R^2$ is a {\it convex polygon} if it is a convex body (it must have non-empty interior) which is the convex hull of finitely many points $q_1, \ldots, q_k$ ordered in such a way that the segments $[q_1, q_2], \ldots [q_{k-1}, q_k]$ and $[q_k, q_1]$ have disjoint relative interiors and their union is the boundary of $K$.
These segments are the {\it sides} of $K$.
The {\it vertices} of $K$ are the points $q_i$, and we do not require them to be extremal points of $K$ (three vertices can be aligned).

When no confusion arises, we will identify the sides of $K$ with the corresponding vectors $q_i - q_{i+1}$.

A pair of {\it opposite parallel sides} is a pair of sides of $K$ that are parallel and are not contained in the same line.

Let us denote by $L,R$ (standing for left and right) the $2\times 2$ matrices of $90$ degrees rotations in the counter-clockwise and clockwise directions, respectively.
Clearly $R = -L = L^{-1} = L^t$ and $\langle R x, y \rangle = \langle x, L y \rangle$ for every $x,y \in \R^2$.

From now on, until the end of the paper we will assume that $p \in (-1,0)$.

\section{Polygons generated by alternating vectors}
\label{sec_polygons}

Let $z_1, \ldots, z_m, x$ be non-zero vectors in the plane.
We say that the vectors in $Z = \{z_1, \ldots, z_m\}$ are {\it alternating} if $(-1)^{i+1} \langle L z_{i+1}, z_i \rangle > 0$ for all $i = 1, \ldots, m-1$, and that they are {\it oriented} with $x$ if $\langle z_i, Rx \rangle > 0$ for all $i = 1, \ldots, m$ (see Figure \ref{fig_Kgenerated}).
We say that a convex polygon $K$ is {\it generated} by the set of vectors $Z$ in the direction of $x$ if $Z$ is alternating and oriented with $x$, and the vertices of $K$ are $p_0 = 0$ and $p_i = \sum_{j=1}^i z_j$ for $i = 1, \ldots, m$ (they are not ordered as the $q_i$ before).
We simply say that $K$ is generated by $Z$ if it is generated by $Z$ in some direction $x$.
Notice that two consecutive $z_i$ cannot be parallel.

The purpose of this section is to compute explicitly $\|x\|_\RpK$ in terms of $Z$.
The formula obtained will hold in a cone determined by $Z$, and we will show that $\|\cdot\|_\RpK$ is convex there.

\begin{figure}
	\includegraphics[width=.8\textwidth]{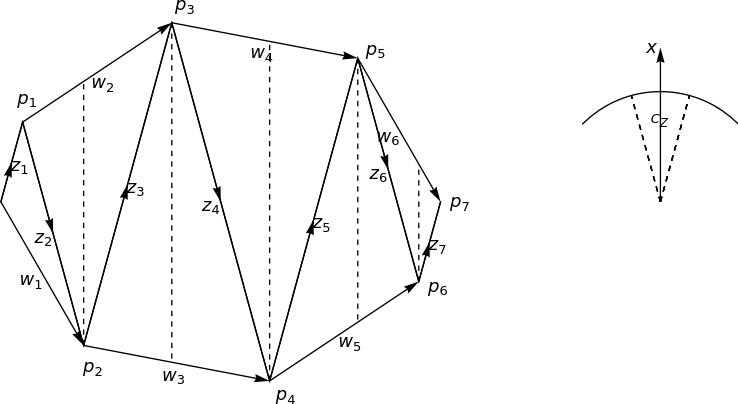}
	\caption{The convex polygon $K$ generated by the vectors $z_1, \ldots, z_m$.}
	\label{fig_Kgenerated}
\end{figure}
The sides of the polygon are $z_1, z_m$ and $w_i$ where
\begin{equation}
	\label{eq_def_wi}
	w_i = z_i + z_{i+1}\text{ for } i=1, \ldots, m-1.
\end{equation}
With this notation the sides $w_i$ with odd $i$ are in one arc of $\partial K$ connecting $0$ with $p_m$, while $z_1$ and the sides $w_i$ with even $i$ are in the opposite arc. The side $z_m$ can belong to either of the two arcs, depending on the parity of $m$.
Notice that $p_i$ are ordered in such a way that they alternate between these two arcs (see Figure \ref{fig_Kgenerated}).

For $i=1, \ldots, m-1$ we define the vectors 
\begin{align}
	\label{eq_def_ni}
	n_i 
	&= (-1)^{i+1} \langle L w_i, z_i \rangle^{-1} L w_i \\
	&= (-1)^{i+1} \langle L z_{i+1}, z_i \rangle^{-1} L w_i \\
	&= (-1)^{i} \langle L w_i, z_{i+1} \rangle^{-1} L w_i \\
	&= (-1)^{i} \langle L z_i, z_{i+1} \rangle^{-1} L w_i,
\end{align}
where the equalities hold thanks to the facts that $w_i = z_i + z_{i+1}$ and $\langle L w_i, w_i \rangle =0$.
The vector $n_i$ is orthogonal to $w_i$ and satisfies $\langle n_i, x \rangle > 0$ and $\langle n_i, z_i \rangle (-1)^{i+1} = \langle n_i, z_{i+1} \rangle (-1)^i = 1$, as it can be seen from the definition and the fact that $Z$ is alternating.

Let 
\begin{equation}
	\label{eq_defC}
	C_Z = \{y \in \R^2 \setminus \{0\}: \langle y, Lz_i \rangle \geq 0, \text{ for } i=1, \ldots, m\},
\end{equation}
and
\begin{equation}
	\label{eq_defCP}
	C_Z' = \{y \in \R^2 \setminus\{0\}: \langle y, n_i \rangle > 0, \text{ for } i=1, \ldots, m-1\},
\end{equation}
then $C_Z, C_Z'$  are two cones (closed and open, respectively) outside the origin, clearly containing the vector $x$.
The contention $C_Z  \subset C_Z'$ is easy to prove. If $y \in C_Z$ then $\langle y, L w_i \rangle = \langle y, L z_i \rangle + \langle y, L z_{i+1} \rangle \geq 0$, with equality if and only if both terms are $0$, and this cannot happen because $z_i, z_{i+1}$ are not parallel and $y \neq 0$.
Clearly $Z$ is oriented with every direction in the interior of $C_Z$.

Let $x \in C_Z$ be a unit vector.
The line passing through any vertex $p_i$ parallel to $x$, $1 \leq i \leq m-1$, touches the side $w_i$ on the opposite arc.
The distance between these two points is given by $X_i = \langle n_i, x \rangle^{-1}$ for every $i = 1, \ldots, m-1$.
This is because $n_i$ is perpendicular to $w_i$, and its norm is the inverse of the distance from $p_i$ to the line containing the side $w_i$.
Thus the X ray in the direction of $v$ between the points $p_i$ and $p_{i+1}$ is an affine function interpolating the values
$X_i$ and $X_{i+1}$. This is still true for $i=0$ and $i = m-1$ if we set $X_0 = X_m = 0$.

Recall that $K$ has no pair of opposite parallel sides and $w_i, w_{i+1}$ belong to two different arcs, so $\{n_{i-1}, n_i\}$ is linearly independent for every $i = 2, \ldots, m-1$ (still, $n_i$ and $n_{i+2}$ are allowed to be parallel).
Define $a_i, \tilde a_i,  b_i, c_i \in \R$ for $i = 2, \ldots, m-1$, by the relations
\begin{align}
	L z_i = a_i n_{i-1} + \tilde a_i n_i \label{eq_def_abc_i} \\
	L z_{i+1} = b_i n_{i-1} + c_i n_i \label{eq_def_abc_ii} .
\end{align}
\begin{proposition}
	\label{prop_coefficients}
	For $i = 2, \ldots, m-1$ the following relations hold:
	\begin{enumerate}[label=\alph*), ref=\alph*]
		\item \label{prop_coefficients_aa} $\tilde a_i = -a_i$
		\item \label{prop_coefficients_ab} $b_i = - a_i$
		\item \label{prop_coefficients_ca} $c_i = (-1)^{i+1} \langle L z_{i+1}, z_i \rangle + a_i$
	\end{enumerate}
\end{proposition}
\begin{proof}
	The first assertion, which is the fact that $L z_i$ is parallel to $n_{i-1} - n_i \neq 0$, comes from the definition of $n_i$ and
\[\langle z_i, n_{i-1} - n_i \rangle = (-1)^i - (-1)^{i+1} = 0.\]

	For the second and third assertions use \eqref{eq_def_wi} and \eqref{eq_def_ni} to write
\begin{align}
L z_{i+1}
	&= L w_i - L z_i \\
	&= (-1)^{i+1} \langle L w_i, z_i \rangle n_i - a_i n_{i-1} + a_i n_i \\
	&= ( (-1)^{i+1} \langle L z_{i+1}, z_i \rangle + a_i ) n_i - a_i n_{i-1}
\end{align}
which shows that $c_i = (-1)^{i+1} \langle L w_i, z_i \rangle + a_i$ and $b_i = - a_i$.
%
%
%
\end{proof}
Define 
\begin{align}
	\label{eq_def_alpha}
	\alpha_1 &= -a_2 + \langle L z_2, z_1 \rangle, \\
	\alpha_i &= b_{i+1} + c_i, \text{ for } 2 \leq i \leq m-2 \\ 
	\alpha_{m-1} &= c_{m-1}
\end{align}
\begin{proposition}
	\label{prop_structure_f}
	Let $K$ be a polygon with no pair of opposite parallel sides, generated by a set of alternating vectors $Z = \{z_1, \ldots, z_m\}$.
	 Then for every $x \in C_Z$ and $p \in (-1,0)$,
	\begin{equation}
		\label{eq_norm_psum}
		\|x\|_\RpK = \left( \frac{p+1}{p+2} \vol K \sum_{i=1}^{m-1} \alpha_i \langle n_i, x \rangle^{-p} \right)^{-1/p}
	\end{equation}
	Moreover, the function 
	\begin{equation}
		\label{eq_def_fZ}
		f_Z(x) = \frac{p+1}{p+2} \vol K \sum_{i=1}^{m-1} \alpha_i \langle n_i, x \rangle^{-p}
	\end{equation}
	is well defined and $C^\infty$ in $C_Z'$.
\end{proposition}
\begin{proof}
	The fact that $f_Z$ is well defined and $C^\infty$ in $C_Z'$ is clear from the definition of $C_Z'$.

	Since the expression on the right-hand side of \eqref{eq_norm_psum} is homogeneous of degree $1$, we may assume that $x$ is unitary.
	By continuity of both sides of \eqref{eq_norm_psum} we may assume $x$ is an interior point of $C_Z$, implying that $Z$ is oriented with $x$.
	The X rays are parametrized by $t \in \R \mapsto t R x + \langle x \rangle$.
	For notational convenience we write $X_x K(t)$ instead of $X_x K(t R x)$.
	Formula \eqref{eq_norm_Xrays} can be split as
	\begin{equation}
		\label{eq_norm_Xrays_pieces}
		\|x\|_\RpK = \left( (p+1) \vol{K} \sum_{i=0}^m \int_{\langle p_i, Rx \rangle}^{\langle p_{i+1}, Rx \rangle} X_x K(t)^{1+p} d t \right)^{-1/p}.
	\end{equation}
	Let $X_i, X_{i+1}$ denote the lengths of the X rays parallel to $x$, going through the points $p_i, p_{i+1}$ respectively.
	The function $X_xK(t)$ is piecewise linear, then the part between $X_i$ and $X_{i+1}$ in the integral \eqref{eq_norm_Xrays} gives
	\begin{align}
		\int_{\langle p_i, Rx \rangle}^{\langle p_{i+1}, Rx \rangle} & X_x K(t)^{1+p} d t \\
		&= \int_{\langle p_i, Rx \rangle}^{\langle p_{i+1}, Rx \rangle} \left(X_i + \frac{t - \langle p_i, Rx \rangle}{\langle p_{i+1}, Rx \rangle - \langle p_i, Rx \rangle} (X_{i+1} - X_i)\right)^{1+p} d t \\
		&= (\langle p_{i+1}, Rx \rangle - \langle p_i, Rx \rangle) \int_0^1 (X_i + s (X_{i+1} - X_i))^{1+p} d t \\
		&= \frac 1{p+2}\langle z_{i+1}, Rx \rangle \frac {X_{i+1}^{2+p} - X_i^{2+p}}{X_{i+1} - X_i} \\
		&= \frac 1{p+2}\langle z_{i+1}, Rx \rangle \left(X_i^{1+p} + X_{i+1}^{1+p} + \frac{X_{i+1}^p - X_i^p}{X_i^{-1} - X_{i+1}^{-1}} \right)
	\end{align}
	Now denote $x_i = \langle n_i, x \rangle$ and recall that $x_i > 0$ for all $x \in C_Z$, then
	\begin{align}
		(p+2) \int_{\langle p_i, Rx \rangle}^{\langle p_{i+1}, Rx \rangle} X_x K(t)^{1+p} d t 
		&= \langle z_{i+1}, Rx \rangle \left( x_i^{-1-p} +x_{i+1}^{-1-p} + \frac{x_{i+1}^{-p} - x_i^{-p}}{x_i - x_{i+1}} \right).
	\end{align}
	Similarly, we get for the first and last intervals,
	\[ (p+2) \int_0^{\langle p_1, Rx \rangle} X_x K(t)^{1+p} d t = \langle L z_1, x \rangle x_1^{-1-p}\]
	and
	\[ (p+2) \int_{\langle p_{m-1}, Rx \rangle}^{\langle p_m, Rx \rangle} X_x K(t)^{1+p} d t = \langle L z_{m}, x \rangle x_{m-1}^{-1-p}.\]

	Putting all the terms together, reordering the sum, using \eqref{eq_def_wi}, \eqref{eq_def_ni}, writing $\langle L z_i, x \rangle = a_i x_{i-1} + \tilde a_i x_i$ and using Proposition \ref{prop_coefficients} item \eqref{prop_coefficients_aa},
	\begin{align}
		\frac {p+2}{(p+1) \vol{K}} & \|x\|_\RpK^{1+p}
		= \int_0^{\langle p_m, Rx \rangle} X_x K(t)^{1+p} d t \\
		&= \frac{\langle L z_1, x \rangle}{ x_1^{1+p} } + \sum_{i=2}^{m-1} \langle L z_i, x \rangle \left( x_{i-1}^{-1-p} +x_i^{-1-p} + \frac{x_i^{-p} - x_{i-1}^{-p}}{x_{i-1} - x_i} \right) + \frac{\langle L z_{m}, x \rangle}{x_{m-1}^{1+p}}  \\
		&= \frac{\langle L w_1, x \rangle}{x_1^{1+p}} + \sum_{i=2}^{m-1} \left( \langle L z_i, x \rangle \frac{x_i^{-p} - x_{i-1}^{-p}}{x_{i-1} - x_i}  + \frac {\langle L w_i, x \rangle}{x_i^{1+p}} \right) \\
		&= \langle L z_2, z_1 \rangle x_1^{-p} + \sum_{i=2}^{m-1} \left( a_i (x_i^{-p} - x_{i-1}^{-p} ) + \langle L z_{i+1}, z_i \rangle (-1)^{i+1} x_i^{-p} \right) \\
		&= \sum_{i=1}^{m-1} x_i^{-p} (- a_{i+1} + a_i + \langle L z_{i+1}, z_i \rangle (-1)^{i+1} ) \\
	\end{align}
	where we take $a_1 = a_m = 0$.

	By Proposition \ref{prop_coefficients} items \eqref{prop_coefficients_ab} and \eqref{prop_coefficients_ca}, we obtain
	\begin{align}
		\frac {p+2}{(p+1) \vol{K}} f_Z(x)
		&= \sum_{i=1}^{m-1} x_i^{-p} (- a_{i+1} + a_i + \langle L z_{i+1}, z_i \rangle (-1)^{i+1} ) \label{eq_norm_psum_a}\\
		&= \alpha_1 x_1^{-p} +  \sum_{i=2}^{m-2} x_i^{-p} (b_{i+1} + c_i) + \alpha_{m-1} x_{m-1}^{-p}, \\
	\end{align}
	and the proposition follows.
\end{proof}

	For later reference we write how $f_Z$ is computed for every $x \in C_Z'$,
	\begin{align}
		\label{eq_decomposition_X}
		\frac {p+2}{(p+1) \vol{K}} f_Z(x) 
		&= \langle L z_1, x \rangle \int_0^1 (s \langle x, n_1 \rangle)^{1+p} dt \\
		&+ \sum_{i=1}^{m-1} \langle z_i, Rx \rangle \int_0^1 (\langle x, n_i \rangle + s (\langle x, n_{i+1} \rangle - \langle x, n_i \rangle))^{1+p} d t  \\
		&+ \langle L z_m, x \rangle \int_0^1 (s \langle x, n_{m-1} \rangle)^{1+p} dt.
	\end{align}

\begin{proposition}
	\label{prop_signs}
	Let $K$ be a polygon with no pair of opposite parallel sides, generated by a set of alternating vectors $Z = \{z_1, \ldots, z_m\}$.
	Then there is $i_0 \in \{1, \ldots, m-1\}$ such that $\alpha_{i_0} > 0$ while $\alpha_i \leq 0$ for every $i\neq i_0$.
\end{proposition}
\begin{proof}
	Here we will define $w_0 = z_1$ and $w_m = z_m$.
	With this convention, the sides of $K$ are exactly $w_0, \ldots, w_m$.
	The convexity of $K$ implies the inequality
	$ \langle w_{i+1}, (-1)^{i} L w_{i-1} \rangle \geq 0$ for all $i = 0, \ldots, m-1$.

	By the convexity of $K$ and the fact that $K$ has no pair of opposite parallel sides, there exists exactly one index $i_0 \in \{1, \ldots, m-1\}$ for which $\langle w_i, L w_{i-1} \rangle$ and $\langle w_i, L w_{i+1} \rangle$ have different signs where these numbers are non-zero for every $i = 1, \ldots, m-1$.
	We claim that for every $i \in \{1, \ldots, m-1\}$, the sign of $\alpha_i$ is that of $-\langle w_i, L w_{i-1} \rangle \langle w_i, L w_{i+1} \rangle$, then the result follows.

	For two linearly independent vectors $u,v \in \R^2$ denote $P_{u,v}$ the linear projection with image $\langle v \rangle$ and kernel $\langle u \rangle$.
	Then clearly $P_{u,v}(x) = v \frac{\langle x, u^\perp \rangle}{\langle v, u^\perp \rangle}$ and $P_{u,v}(x) + P_{v,u}(x) = x$.

	Let us prove the claim for $i=1$. By \eqref{eq_def_abc_i} and \eqref{eq_def_wi},
	\begin{align}
		z_2 &= a_2 R n_1 - a_2 R n_2 \\
		z_2 &= -w_0 + w_1\\
	\end{align}
	which implies that
	\begin{align}
		P_{w_2, w_1}(z_2) - P_{w_0, w_1}(z_2)
		&= (a_2 \langle L w_1, z_1 \rangle^{-1} - 1)w_1 \\
		&= -\alpha_1 \langle L z_2, , z_1 \rangle^{-1} w_1.
	\end{align}
	Since $P_{w_2, w_1}(w_1) - P_{w_0, w_1}(w_1) = w_1 - w_1 = 0$ and $P_{w_0, w_1}(z_1)=0$, we have
	\[
		-P_{w_2, w_1}(z_1) = -\alpha_1 \langle L z_2, , z_1 \rangle^{-1} w_1,
	\]
	and using the formula for the projection we get
	\[
		\frac {\langle z_1, n_2 \rangle}{\langle w_1, n_2 \rangle} = \alpha_1 \langle L z_2, , z_1 \rangle^{-1}\\
	\]
	but $\langle z_1, n_2 \rangle = \langle w_0, n_2 \rangle \geq 0$, $\langle w_1, L w_0 \rangle = -\langle z_1, L z_2 \rangle < 0$, and the sign of $\langle w_1, n_2 \rangle$ is the same as that of $\langle w_1, L w_2 \rangle$.
	So either $\alpha_1=0$ or the sign of $\alpha_1$ is the same as that of $-\langle w_1, L w_2 \rangle \langle w_1, L w_0 \rangle$.

	A similar computation works for $i=m-1$.

		Now for $i = 2, \ldots, m-1$ we write $z_{i+1}$ in two different ways using \eqref{eq_def_abc_i} with $z_{i+1}$ instead of $z_i$, and \eqref{eq_def_abc_ii}. By \eqref{eq_def_ni} and Proposition \ref{prop_coefficients} items \eqref{prop_coefficients_aa} and \eqref{prop_coefficients_ab} one has
		\begin{align}
			z_{i+1}
			&= b_{i+1} (-1)^{i+2} \langle L w_{i+1}, z_{i+1} \rangle^{-1} w_{i+1} - b_{i+1} (-1)^{i+1} \langle L w_i, z_i \rangle^{-1} w_i \\
			&= b_i (-1)^{i-1} \langle L w_{i-1}, z_i \rangle^{-1} w_{i-1} + c_i (-1)^{i+1} \langle L w_i, z_i \rangle^{-1} w_i.
		\end{align}
		This means that
		\begin{align}
			 (-1)^{i+1} \alpha_i \langle L z_{i+1}, z_i \rangle^{-1} w_i
			 &= P_{w_{i-1}, w_i}(z_{i+1}) - P_{w_{i+1},w_i}(z_{i+1})\\
			 &= - P_{w_{i-1}, w_i}(z_i) + P_{w_{i+1},w_i}(z_i) \label{eq_coefficient_geometry}\\
		\end{align}
	where we used again that $P_{w_{i+1},w_i}(w_i) - P_{w_{i-1}, w_i}(w_i) = w_i - w_i = 0$.

	By the formula for $P_{u,v}$, 
	\begin{align}
		(-1)^{i+1} \alpha_i \langle L z_{i+1}, z_i \rangle^{-1}
		&= \frac{\langle z_i, n_{i+1} \rangle}{\langle w_i, n_{i+1} \rangle} - \frac{\langle z_i, n_{i-1} \rangle}{\langle w_i, n_{i-1} \rangle} \\
	&= \frac{\langle z_i, n_{i+1} \rangle\langle w_i, n_{i-1} \rangle-\langle w_i, n_{i+1} \rangle\langle z_i, n_{i-1} \rangle	}{\langle w_i, L w_{i-1} \rangle \langle w_i, L w_{i+1} \rangle}\\
	&= \frac{
	\det\left( \left(\begin{array}{c} n_{i+1} \\ \hline  n_{i-1} \end{array} \right) \cdot \left(\begin{array}{c|c} z_i & w_i \end{array} \right) \right)
	}{\langle w_i, L w_{i-1} \rangle \langle w_i, L w_{i+1} \rangle}\\
\end{align}
		which is non-positive because
		\[ (-1)^{i+1} \det \left(\begin{array}{c} n_{i+1} \\ \hline  n_{i-1} \end{array} \right) =  \langle w_{i+1}, (-1)^{i} n_{i-1} \rangle (-1)^{i+2} \langle L z_{i+2}, z_{i+1} \rangle \geq 0 \]
	\[ (-1)^{i+1} \det \left(\begin{array}{c|c} z_i & w_i \end{array} \right) = (-1)^{i+1} \langle L z_i, z_{i+1} \rangle < 0.\]

		Finally, we must show the strict inequality $\alpha_{i_0} > 0$.
		The quantity $(-1)^{i+1} \langle L z_{i+1}, z_i \rangle$ is twice the area of the triangle with vertices $p_i, p_{i-1}, p_{i+1}$, and the non-overlaping union of these triangles is $K$.
		By formula \eqref{eq_norm_psum_a}, 
		\begin{align}
			\sum_{i=1}^{m-1} \alpha_i
			&= \sum_{i=1}^{m-1} (-1)^{i+1} \langle L z_{i+1}, z_i \rangle \\
			&= 2 \vol{K} >0.
		\end{align}
		This proves the claim.

\end{proof}

	\begin{figure}
		\includegraphics[width=.9\textwidth]{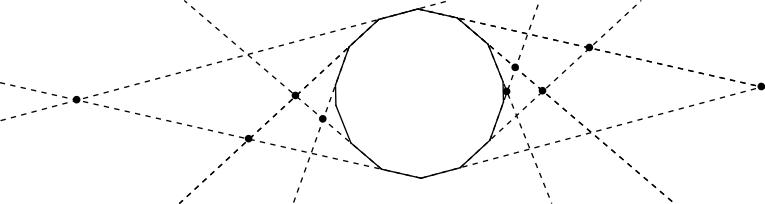}
		\caption{The coefficient $\alpha_i$ is the area of a parallelogram with sides $[r_i, r_{i+1}]$ and $z_i$.}
		\label{fig_intersectionpoints}
	\end{figure}

\begin{remark}
	Proposition \ref{prop_signs} is the place where convexity is really used.
	Equation \eqref{eq_coefficient_geometry} offers a geometric insight of the meaning of the coefficients $\alpha_i$.
	If we denote by $r_i$ the intersection of the two lines containing the sides $w_{i-1}$ and $w_i$, then $P_{w_{i+1},w_i}(z_i) - P_{w_{i-1}, w_i}(z_i)$ equals $r_{i+1} - r_i$.
	The fact that $r_i$ ``moves to the left'' (in the direction of $Lx$) when $i$ runs from $1$ to $i_0-1$, is due to the convexity of $K$.
	This point ``comes back'' from the right, exactly when taking the intersection with $w_{i_0}$, and keeps moving to the left for $i$ between $i_0+1$ and $m-1$ (see Figure \ref{fig_intersectionpoints}).
\end{remark}

\begin{proposition}
	\label{prop_convexincone}
	Let $K$ be a polygon with no pair of opposite parallel sides, generated by a set of alternating vectors $Z = \{z_1, \ldots, z_m\}$.
	Then for every $p \in (-1,0)$, the function $f_Z^{-1/p}$ is convex in an open cone containing $C_Z$.
\end{proposition}
\begin{proof}
	By Propositions \ref{prop_structure_f} and \ref{prop_signs}, $f_Z$ can be expressed as
	\[f_Z(x) = \langle v_0, x \rangle^{-p} - \varphi(x) \]
	where either $\varphi \equiv 0$ or
	\[\varphi(x) = \sum_{j=1}^{k} \langle x, v_j \rangle^{-p},\]
	$v_j \in \R^2 \setminus \{0\}$
	and $\langle v_j, \cdot \rangle$ are positive functions in $C_Z'$ for $j \geq 0$.

	In the first case $f_Z^{-1/p}$ is linear in $C_Z$ and there is nothing to prove, so let us assume $\varphi$ is not identically zero.

	Let us show that $\varphi(x)^{-1/p}$ is concave.
	For every $x,y \in \R^2, \lambda \in [0,1]$,
	\begin{align}
		\varphi(\lambda x + (1-\lambda) y)
		&= \left( \sum_{i=1}^{k} (\lambda \langle x, v_i \rangle + (1-\lambda) \langle y, v_i \rangle)^{-p} \right)^{-1/p} \\
		&\geq \lambda \left( \sum_{i=1}^{k} \langle x, v_i \rangle^{-p}\right)^{-1/p} + (1-\lambda) \left( \sum_{i=1}^{k} \langle y, v_i \rangle^{-p} \right)^{-1/p}\\
		&= \lambda \varphi(x) + (1-\lambda) \varphi(y),
	\end{align}
	where we used that $(x_1, \ldots, x_k) \mapsto \left(\sum_i x_i^{-p} \right)^{-1/p}$ is a concave function in the positive orthant, and $\lambda \langle x, v_i \rangle + (1-\lambda) \langle y, v_i \rangle > 0$ for all $i$.

	The functions $f_Z$ and $\varphi$ are strictly positive in $C_Z$ so we may take a closed cone outside the origin $D \subseteq C_Z'$ containing $C_Z$ in its interior, such that $f_Z$ and $\varphi$ are strictly positive in $D$.

	To show that $f_Z^{-1/p}$ is convex, we observe that it is positive and homogeneous of degree $1$ in $C_Z'$, so that it suffices to show that the level set $U = \{x \in D: f_Z(x)^{-1/p} \leq 1\}$ is convex.

	The set $U$ is closed and star-shaped with respect to the origin, so it is connected.
	By Lemma \ref{lem_locallyconvex}, it is enough to show that every point in $U$ is inside a small ball $B$ such that $B \cap U$ is convex.
	Since $D$ is a closed cone outside the origin, this is evident for every point in $\{x \in D: f_Z(x)^{-1/p} < 1\}$. Then it suffices to check this condition for every point in the surface $S = \{x \in D: f_Z(x)^{-1/p} = 1\}$.

	Take any $x_0 \in S$, so that $f_Z(x_0) = \langle x_0, v_0 \rangle^{-p} - \varphi(x_0) = 1$ and $\langle x_0, v_0 \rangle^{-p} - 1 = \varphi(x_0) > 0$. 
	Let $\varepsilon > 0$ be small enough so that for every $x \in B(x_0, \varepsilon) \cap D$, we have $\langle x, v_0 \rangle^{-p} > 1$.

	For every $x \in B(x_0, \varepsilon) \cap D$, the following are equivalent:
	\begin{align}
		f_Z(x) = \langle v, x \rangle^{-p} - \varphi(x) &\leq 1 \\
		\varphi(x) &\geq \langle v, x \rangle^{-p} - 1 \\
		\varphi(x)^{-1/p} &\geq ( \langle v, x \rangle^{-p} - 1)^{-1/p} \\
		\varphi(x)^{-1/p} - ( \langle v, x \rangle^{-p} - 1)^{-1/p} &\geq 0.
	\end{align}
	Observe that $t\mapsto (t^{-p}-1)^{-1/p}$ is a convex function for $t>1$, so the right-hand side is a concave function of $x$.
	Then we deduce that $U \cap B(x_0, \varepsilon)$ is a convex set, and the theorem follows.
\end{proof}

\section{Proof of Theorem \ref{thm_main}}
\label{sec_cones}
In this section we will extend the convexity of $\|\cdot\|_\RpK$ from one cone to the whole plane.

	Let $K$ be a convex polygon with no pair of opposite parallel sides and let $p_i$ be the vertices.
	Consider the set of vectors $p_i - p_j, 1 \leq i,j \leq m, i\neq j$.
	The lines generated by these vectors divide the plane into some finite collection of non-overlapping cones $C_1(K), \ldots, C_k(K)$.

\begin{proposition}
	\label{prop_convexinsomecone}
	In each $C_i(K)$, the function $\|\cdot\|_\RpK$ is convex for every $p \in (-1,0)$.
\end{proposition}
\begin{proof}
	Let $x$ be in the interior of $C_i(K)$. Since $x$ is not parallel to any $p_i - p_j$, the numbers $\langle R x, p_i \rangle$ are all different.
	Now we may translate $K$ and if necessary add vertices to the sides of $K$ to find a set of alternating vectors $Z$ generating $K$, which are oriented with $x$.
	Notice that adding vertices to the sides of $K$ do not make pairs of opposite parallel sides appear, and translating $K$ does not change $\RpK$.

	By Proposition \ref{prop_convexincone}, $\|\cdot \|_\RpK$ is a $C^2$ convex function in an open cone that contains $x$ in the interior (this new cone might be smaller than $C_i(K)$ since we added vertices to $K$).
This reasoning can be applied to an arbitrary point in $C_i(K)$. Since $\|\cdot\|_\RpK$ is $C^2$ in the interior of $C_i(K)$, we deduce that for every $x$ in the interior of $C_i(K)$ the Hessian of $\|\cdot \|_\RpK$ is positive semi-definite.
	Then $\|\cdot \|_\RpK$ is convex in the whole $C_i(K)$.
\end{proof}

Now we must prove that $\|\cdot\|_\RpK$ is convex in the intersection of every pair of cones.
A direction $x\neq 0$ is in the intersection of two cones if it is parallel to some $p_i - p_j$.
We must distinguish two cases: If $x$ is parallel to a side we shall see that $\|\cdot\|_\RpK$ has a convex vertex.
Otherwise we will see that $\|\cdot\|_\RpK$ the tangents of $R_pK$ on both sides of the cone, coincide (the body is $C^1$ but not $C^2$ at this point).

\begin{proposition}
	\label{prop_cone_union_vertex}
	Assume $K$ has no pair of opposite parallel sides and let $p \in (-1,0)$.
	If $x_0$ is parallel to a side of $K$ then there is a neighbourhood of $x_0$ where the restriction of $\|\cdot \|_\RpK$ is convex.
\end{proposition}
\begin{proof}
	The line containing the side parallel to $x_0$ determines two open half-planes $D_+, D_-$. Let $D_+$ be the one containing the interior of $K$.
	First choose alternating vectors $Z = \{z_1, \ldots, z_m\}$ generating $K$, in such a way that $z_1$ is the (unique) side of $K$ parallel to $x_0$, and that $Z$ is oriented with all $x \in D_- \cap B$ where $B$ is a small ball centered in $x_0$ (Figure \ref{fig_side}).
	By Proposition \ref{prop_convexincone} the function $f_Z$ is well defined, smooth and convex in an open cone $D'$ containing $x_0 \in \partial C_Z$, and equals $\|\cdot\|_\RpK$ in $D' \cap D_- \cap B$.

	\begin{figure}
	\includegraphics[width=.45\textwidth]{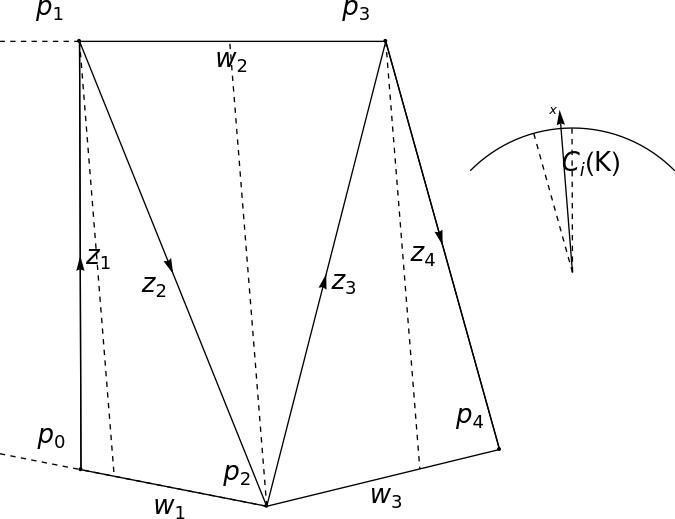}
	\includegraphics[width=.45\textwidth]{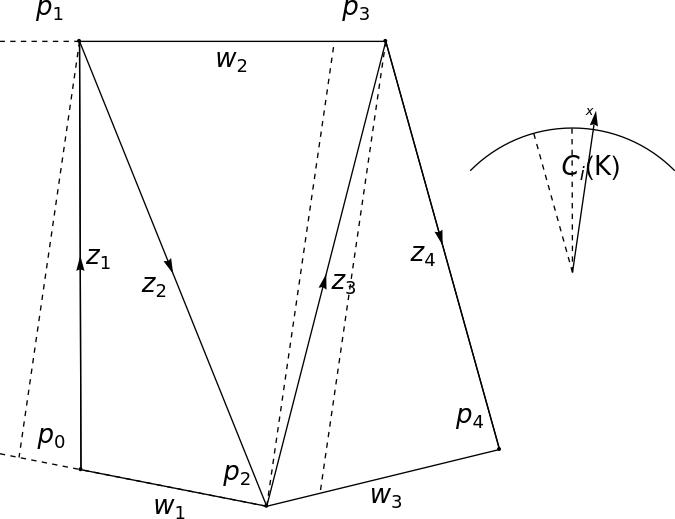}
		\caption{}
		\label{fig_side}
	\end{figure}

	Recall that equation \eqref{eq_decomposition_X} computes the integral of the linear interpolations of the X rays going through the vertices, computed as $X_i = \langle n_i, x \rangle^{-1}$.
	We will analyze the values of $f_Z(x)$ when $x$ is in $D' \cap D_+ \cap B$.
	For $x$ in this set the term $\langle n_1, x \rangle^{-1}$ no longer measures the lengths of the X rays through $K$.
	Instead, we must define two auxiliary lengths, $Y_x K(t), Z_x K(t)$.
	Consider all the lines parallel to $x$ going through points $t Rx \in \langle x \rangle^\perp$.
	For $t Rx + \langle x \rangle$ between $p_0$ and $p_1$, call $Y_x K(t)$ the length of the segment in the line $t Rx + \langle x \rangle$ from the side $z_1$ to the line containing the side $w_1$.
	Likewise, call $Z_x K(t)$ the length of the segment in the line $t Rx + \langle x \rangle$ from the line containing the side $w_1$ to the side $w_2$.
	We see that $Y_x K(t)$ is the linear interpolation between $0$ and $\langle n_1, x \rangle^{-1}$ when $t$ goes from $\langle p_1, Rx \rangle$ to $\langle p_0, Rx \rangle$. Similarly $Z_x K(t)$ is the linear interpolation between $\langle n_1, x \rangle^{-1}$ and $\langle n_2, x \rangle^{-1}$ when $t$ goes from $\langle p_0, Rx \rangle$ to $\langle p_2, Rx \rangle$.
	By shrinking $B$ if necessary, one may assume that the lines parallel to $x$ passing through all the other vertices $p_i$ intersect the side $w_i$ in the relative interior (here we use that no other vector $p_i - p_j$ is parallel to $x_0$).
	This means that equation \eqref{eq_decomposition_X} becomes (notice that the term $\langle L z_1, x \rangle$ is negative)
	\begin{align}
		\frac 1{(p+1) \vol{K}} f_Z(x) 
		&= -\int_{\langle p_1, R v \rangle}^{\langle p_0, R v \rangle} Y_x K(t)^{1+p} d t + \int_{\langle p_1, R v \rangle}^{\langle p_0, R v \rangle} Z_x K(t)^{1+p} d t \\& + \int_{\langle p_0, R v \rangle}^\infty X_x K(t)^{1+p} d t\\
		&= \int_{\langle p_1, R v \rangle}^{\langle p_0, R v \rangle} (Z_x K(t)^{1+p} - X_x K(t)^{1+p} - Y_x K(t)^{1+p}) d t \\& + \int_{\langle p_1, R v \rangle}^\infty X_x K(t)^{1+p} d t \leq \frac 1{(p+1) \vol{K}} \|\cdot \|_\RpK^{-p}
	\end{align}
	where we used that $Z_x K(t)^{1+p} = (X_x K(t)+Y_x K(t))^{1+p} \leq X_x K(t)^{1+p} + Y_x K(t)^{1+p}$.
	
	By shrinking $B$ further if necessary, one may assume that $\|x\|_\RpK$ is smooth and convex in $B \cap D_+$.
	To see the convexity in $B$, just observe that for $x,y\in B$ the function $\lambda \in [0,1] \mapsto \|\lambda x + (1-\lambda) y\|_\RpK$ is $C^2$ with non-negative second derivative in all the interval except possibly at a single point (the crossing point between $D_+$ and $D_-$, if there is a crossing) where the derivative has a positive jump discontinuity, so we see that the derivative must be non-decreasing.

\end{proof}

\begin{proposition}
	\label{prop_cone_union_tangent}
	Let $K$ be a convex polygon with no pair of opposite parallel sides and let $p_i$ be the vertices.
	Assume that the vectors $p_i - p_j$ are pairwise not parallel.

	If $x_0$ is parallel to a vector $p_i - p_j$ but not parallel to a side of $K$, then for every $p \in (-1,0)$,
	$\|\cdot\|_\RpK$ is $C^1$ at $x_0$.
	In particular $\|\cdot\|_\RpK$ is convex in a small ball centered at $x_0$.
\end{proposition}
\begin{proof}
	As in the proof of Proposition \ref{prop_cone_union_vertex}, the line parallel to $x_0$ containing the points $p_i, p_j$,  determines two open half-planes $D_+, D_-$.
	As before, by translating $K$ and eventually adding vertices to the sides, we may generate $K$ with alternating vectors $Z = \{z_1, \ldots, z_m\}$ in such a way that $x_0$ is parallel to one of the $z_i$, and $Z$ is oriented with all $x$ which are in $D_+ \cap B$ where $B$ is a small ball centered at $x_0$ (see Figure \ref{fig_fakefunctions}).
	By Proposition \ref{prop_convexincone}, the function $f_Z^{-1/p}$ is defined in an open cone $C'$ containing $x_0$ where it is convex, and coincides with $\|\cdot\|_\RpK$ in $C' \cap D_+$.
	Let $x \in C' \cap D_- \cap B$.
	The value of $f_Z(x)^{-1/p}$ no longer coincides with $\|x\|_\RpK$, but $f_Z^{-1/p}$ is still convex at $C' \cap D_-$.
	We shall prove that $f_Z(x)^{-1/p}$ and $\|x\|_\RpK$ coincide up to second order terms.

		\begin{figure}
	\includegraphics[width=.49\textwidth]{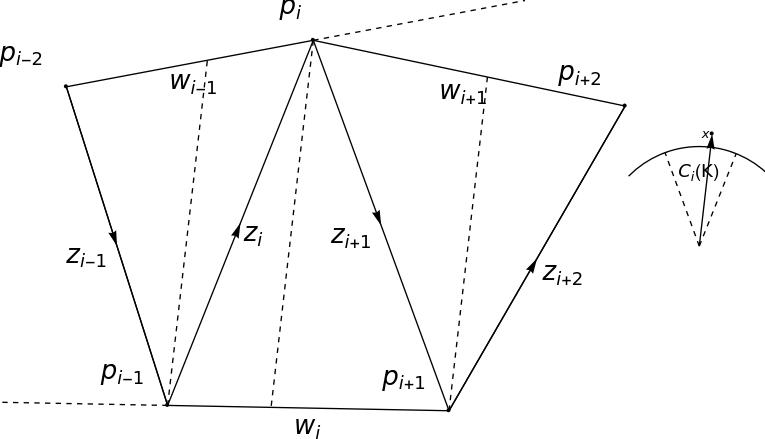}
	\includegraphics[width=.49\textwidth]{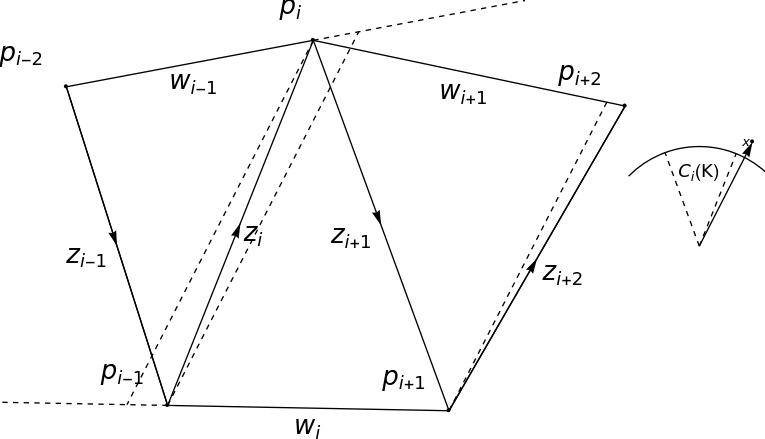}
			\caption{The terms $\langle n_i, x \rangle^{-1}$ give the lengths of the X rays when $x \in C_i(K)$ (left), and are the extended rays $Y,Z$ or $W$ when $x \in C' \setminus C_i(K)$ (right). }
			\label{fig_fakefunctions}
		\end{figure}

	Consider the X rays parallel to $x$.
	Since $x$ is not in $D_+$, the vectors in $Z$ are not oriented with $x$.
	By eventually shrinking $B$ if necessary we may assert that no vector $z_i$ is parallel to $x$ for $x \in D_- \cap B$, then all X rays going through $p_k$ for $k \neq i \geq 2$ intersect the opposite side $w_k$ in its relative interior, and formula \eqref{eq_decomposition_X} is applicable to this decomposition, except for the X rays between $p_{i-2}$ and $p_{i+1}$ (see Figure \ref{fig_fakefunctions}).
	Now consider all the lines parallel to $x$ going through some point $y \in \langle x \rangle^\perp$.
	By analyzing formula \eqref{eq_decomposition_X}, we see that the numbers $\langle n_i, x \rangle^{-1}$ no longer measure the X rays parallel to $x$, and as in the proof of Proposition \ref{prop_cone_union_vertex} we need auxiliary quantities, that are the ones appearing in formula \eqref{eq_decomposition_X}.
	For $t Rx+\langle x \rangle$ between $p_{i-2}$ and $p_{i-1}$, call $Y_x K(t)$ the length of the segment contained in $t Rx + \langle x \rangle$ from $w_{i-2}$ (or $z_{i-1}$ if $i=2$) to the line containing the segment $w_{i-1}$.
	For $t Rx+\langle x \rangle$ between $p_i$ and $p_{i-1}$ call $Z_x K(t)$ the length of the segment contained in $t Rx + \langle x \rangle$ from the line containing the side $w_{i-1}$ to the line containing the side $w_i$.
	For $t Rx+\langle x \rangle$ between $p_i$ and $p_{i+1}$ call $W_x K(t)$ the length of the segment contained in $t Rx + \langle x \rangle$  from the line containing the side $w_i$ to the side $w_{i+1}$.
	Formula \eqref{eq_decomposition_X} now becomes
	(notice that the term $\langle R x, z_i \rangle$ is negative)
	\begin{align}
		\frac {p+2}{(p+1) \vol{K}}  f_Z(x)
		&= \int_{-\infty}^{\langle Rx, p_i \rangle} X_x K(t)^{1+p} d t + \int_{\langle Rx, p_i \rangle}^{\langle Rx, p_{i-1} \rangle} Y_x K(t)^{1+p} d t \\
		&-\int_{\langle Rx, p_i \rangle}^{\langle Rx, p_{i-1} \rangle} Z_x K(t)^{1+p} d t + \int_{\langle Rx, p_i \rangle}^{\langle Rx, p_{i-1} \rangle} W_x K(t)^{1+p} d t \\& + \int_{\langle Rx, p_{i-1} \rangle}^\infty X_x K(t)^{1+p} d t \\
		&= \int_{-\infty}^\infty X_x K(t)^{1+p} d t + \int_{\langle Rx, p_i \rangle}^{\langle Rx, p_{i-1} \rangle} (Y_x K(t)^{1+p} - Z_x K(t)^{1+p} \\& +  W_x K(t)^{1+p} - X_x K(t)^{1+p} ) d t.
	\end{align}
	It is clear that $|{\langle Rx, p_i \rangle} - {\langle Rx, p_{i-1} \rangle}|, |Y_x K(t) - Z_xK(t)|$ and $|W_xK(t) - X_xK(t)|$ are of order $O(|x-x_0|)$ as $x \to x_0$.

	We obtain for $x \in C' \cap D_- \cap B$,
	\[|f_Z(x)^{-1/p} - \|x\|_\RpK| \leq O(|x-x_0|^2)\]
	while for $x \in C' \cap D_+$, $f_Z(x)^{-1/p} = \|x\|_\RpK$.

	This implies that the tangent space to the graph of $f_Z^{-1/p}$ at $x_0$, is also tangent to the graph of $\|x\|_\RpK$.
	Furthermore, $\|\cdot\|_\RpK$ is $C^1$ up to the boundary in each $B \cap D_\pm$.
	Then $\|\cdot\|_\RpK$ is also $C^1$ at $x_0$ too.

	To see the convexity in $B$, just observe that for $x,y\in B$ the function $\lambda \in [0,1] \mapsto \|\lambda x + (1-\lambda) y\|_\RpK$ is $C^2$ with non-negative second derivative in all the interval except possibly at a single point (the crossing point between $D_+$ and $D_-$, if there is a crossing) where it is $C^1$, so the derivative must be non-decreasing.
\end{proof}

Finally we are in conditions to prove the main theorem.
\begin{proof}[Proof of Theorem \ref{thm_main}]
	First assume that $K$ is a polygon with vertices $p_i$ such that all the vectors $p_i - p_j$ are pairwise not parallel.
	We will use Lemma \ref{lem_locallyconvex} with the set 
	\[\RpK = \{x \in \R^n : \|x\|_\RpK \leq 1\},\]
	which is closed and star-shaped with respect to the origin, and thus connected.
	If $\|x\|_\RpK < 1$ then $x$ is in the interior of $\RpK$ and $x$ is inside a (convex) closed ball inside $\RpK$.
	If $\|x\|_\RpK = 1$ then in particular $x \neq 0$. By Propositions \ref{prop_cone_union_vertex} and \ref{prop_cone_union_tangent}, there is a ball $B$ centered at $x$ where $\|\cdot\|_\RpK$ is convex, then $B \cap \RpK$ is convex.
	Lemma \ref{lem_locallyconvex} then implies that $\RpK$ is convex.

	If $K$ is any convex body, take a sequence of polygons $K_m$ as before, converging to $K$ in the Hausdorff metric.
	By Lemma \ref{lem_approximation}, $\|\cdot\|_{R_pK_m}$ converges pointwise to $\|\cdot\|_\RpK$ in $\R^2$.
	The convexity of $\|\cdot\|_{R_pK_m}$ for all $m$ then implies that $\|\cdot\|_\RpK$ is also convex.
	To see this, take $x,y \in \R^2$ and $\lambda \in [0,1]$, and write
	\[\|\lambda x + (1-\lambda) y\|_{R_pK_m} \leq \lambda \|x\|_{R_pK_m} + (1-\lambda) \|y\|_{R_pK_m}.\]
	Taking limits as $m\to \infty$ one obtains the same inequality for $\RpK$, and the proof is now complete.
\end{proof}

\section{Concluding remarks}
\label{sec_remarks}

The restriction on the dimension provides a way of generating polytopes where it is possible to compute $f_Z$.
Proving an analog of Proposition \ref{prop_structure_f} for higher dimensions, for example $n=3$, seems to be significantly more complicated but not impossible, although probably there are better ways to prove the convexity of $\RpK$.
For the moment this appears to be out of reach.

The special form of $f_Z$ in Proposition \ref{prop_structure_f} and the signs of the coefficients given in Proposition \ref{prop_signs} show that there is a specific structure that makes $-p$-combinations of linear functions, convex in a cone.
It would be interesting to find general necessary and sufficient conditions on $n_i, \alpha_i$ that ensure that a function of the form \eqref{eq_norm_psum} is convex.

A second interesting problem is to analyze the limit of formula \eqref{eq_norm_psum} when the polygon approaches a smooth convex body $K$, and find an analog formula for a smooth convex body based on differential invariants of $K$ (tangent plane, curvature, etc). 
As can be observed from the proof of Proposition \ref{prop_signs}, the formula for $f_Z$ does not depend on $Z$ but rather on $K$, since $\alpha_i = 0$ if $w_{i-1}$ and $w_{i+1}$ are parallel.

Lastly, we point out that there exists a similar problem regarding convex norms that seems to be related to our work. Consider the following invariant $p$-norm in the space of $n \times n$ matrices.
\[\|A\|_p = \left( \int_{S^{n-1}} |A.v|^p d \sigma(v)\right)^{1/p},\quad A \in \operatorname{M}_{n,n}(\R)\]
where $\sigma$ is the invariant probability measure of the sphere and $|\cdot|$ is the euclidean norm.
The norm $\|\cdot\|_p$ interpolates between the operator norm when $p \to \infty$ and a multiple of the Hilbert-Schmidt norm when $p=2$.
Also, if $A$ is non-singular and $p=-n$, it equals $|\det(A)|^{-1/n}$, while for $p \to -\infty$ it recovers the smallest singular value.
The function $\|\cdot\|_p$ is clearly convex for $p\geq 1$, since $A \mapsto |A.v|^p$ is convex for every fixed $v \in S^{n-1}$, and thus the level set $\{A \in \operatorname{M}_{n,n}(\R) : \int_{S^{n-1}} |A.v|^p d \sigma(v) \leq 1\}$ is convex.

Since $\|\cdot\|_p$ is unitarily invariant, it is determined by its restriction to the subspace of diagonal matrices. Identifying this subspace with $\R^n$ one has
\[\|x\|_p = \left( \int_{S^{n-1}} \left( \sum_{i=1}^n (x_i v_i)^2 \right)^{p/2} d\sigma(v) \right)^{1/p}.\]
Interestingly enough, numerical simulations suggest that $\|\cdot\|_p$ is also convex for $p \in (0,1)$.
The case $n=2$ shows similarities with formula \eqref{eq_norm_psum}, but without negative coefficients.

\bibliographystyle{abbrv}
\bibliography{../references }

\end{document}